\documentclass[a4paper,11pt,reqno]{article}

\usepackage{hyperref}       
\usepackage{url}            
\usepackage{amsthm}
\usepackage{amsfonts}
\usepackage{amsmath}
\usepackage{amssymb}
\usepackage{color}
\usepackage{multirow}
\usepackage{float}
\usepackage[noadjust]{cite}
\usepackage{mathtools}
\usepackage{blkarray}
\usepackage{booktabs}
\usepackage{hhline}
\usepackage{float}
\usepackage{caption}
\usepackage{faktor}
\usepackage{longtable}
\usepackage{array}
\usepackage[margin=2.7cm]{geometry}
\usepackage{enumerate}
\newcolumntype{C}[1]{>{\centering\arraybackslash}m{#1}}
\usepackage[normalem]{ulem}

\theoremstyle{definition}

\newtheorem{theorem}{Theorem}[section]
\newtheorem{corollary}[theorem]{Corollary}
\newtheorem{lemma}[theorem]{Lemma}
\newtheorem{definition}[theorem]{Definition}
\newtheorem{proposition}[theorem]{Proposition}
\newtheorem{remark}[theorem]{Remark}

\newtheorem{example}[theorem]{Example}

 \usepackage{color}
\usepackage{tikz, framed}
\usetikzlibrary{arrows,shapes,matrix,decorations.pathmorphing}
\usetikzlibrary{backgrounds}

\newcommand\qbin[3]{\left[\begin{matrix} #1 \\ #2 \end{matrix} \right]_{#3}}

\makeindex

\begin{document}

\title{Constructions of new Matroids and Designs over~$\mathbb{F}_q$}
\author{Eimear Byrne \and Michela Ceria \and Sorina Ionica \and Relinde Jurrius \and Elif Sa\c{c}\i kara}

\maketitle

\abstract{A perfect matroid design (PMD) is a matroid whose flats of the same rank all have the same size. In this paper we introduce the $q$-analogue of a PMD and its properties. In order to do so, we first establish a new cryptomorphic definition for $q$-matroids. We show that $q$-Steiner systems are examples of $q$-PMD's and we use this $q$-matroid structure to construct subspace designs from $q$-Steiner systems. 
{We apply this construction to the only known $q$-Steiner system, which has parameters $S(2,3,13;2)$, and hence establish the existence of a new subspace design with parameters $2$-$(13,4,5115;2)$.}
}

\section{Introduction}
In combinatorics, we often describe a $q$-analogue of a concept or theory to be any generalization that replaces finite sets by finite dimensional vector spaces.
Two classical topics in combinatorics that have recently been studied as $q$-analogues are matroids and designs. These objects and some of the connections between them are the main focus of this paper.

A subspace design (also called a $q$-design, or a design over
$\mathbb{F}_q$) is a $q$-analogue of a design. A $t$-$(n,k,\lambda;q)$ subspace design is a collection $\mathcal{B}$ of $k$-dimensional subspaces of an $n$-dimensional
$\mathbb{F}_q$-vector space $V$ with the property that every $t$-dimensional subspace of $V$ is contained in exactly $\lambda$ of the members of $\mathcal{B}$. 
Explicit constructions of subspace designs have proved so far to be more elusive than their classical counterparts. Early papers by Thomas, Suzuki, and Itoh have provided some examples of infinite families of subspace designs \cite{thomas,suzuki2design92,itoh}, while in \cite{largesets} an approach to the problem using {\em large sets} is given. A $q$-analogue of the Assmus-Mattson theorem gives a general construction of subspace designs from coding theory \cite{byrne2019assmus}. 
Further sporadic examples have been found by assuming a prescribed automorphism group of the subspace design \cite{braun2018q}.
For the special case $\lambda=1$ we call such a design a $q$-Steiner system and write $S(t,k,n;q)$. The actual existence of an $S(t,k,n;q)$ Steiner system for $t>1$, was established for the first time when $S(2,3,13;2)$ designs were discovered by Braun \emph{et al}~\cite{qSteiner_exists_2016}. 
No other examples have been found to date. The smallest open case is that of the $S(2,3,7;q)$ Steiner system, also known as the $q$-analogue of the Fano plane.

While subspace designs have been intensively studied over the last decade \cite{braun2018q}, $q$-analogues of matroids have more recently appeared in the literature \cite{JP18,gorla2019rank}.
In fact, the $q$-matroid defined in \cite{JP18} was a re-discovery of a combinatorial object already studied by Crapo~\cite{crapo1964theory}. Classical matroids are a generalisation of several ideas in combinatorics, such as independence in vector spaces and trees in graph theory. One of the important properties of matroids is that there are equivalent, yet seemingly different ways to define them: in terms of their independent sets, flats, circuits, bases, closure operator and rank function. We call these equivalent definitions {\em cryptomorphisms}. Cryptomorphisms for $q$-matroids between independent subspaces, the rank function, and bases were established in ~\cite{JP18}. In \cite{BCJ17} the cryptomorphism via bi-colouring of the subspace lattice is discussed. In \cite{bcj} several cryptomorphisms were shown to hold, namely those with respect to dependent spaces, circuits, the closure function, hyperplanes, open spaces etc. In this paper we also give a cryptomorphic description of a $q$-matroid in terms of its flats.
In the classical case, there is a link between designs and matroids, given by the so-called {\em perfect matroid designs} (PMDs). PMDs are  matroids for which flats of the same rank have the same cardinality. They were studied by Murty and others in ~\cite{murty1970equicardinal} and ~\cite{Young}, who showed in particular that  Steiner systems are among the few examples of PMDs and, more importantly, that they could be applied to construct new designs. In this paper we obtain $q$-analogues of some of these results.

First, we extend the theory of $q$-matroids to include a new cryptomorphism, namely that between flats and the rank function.
We apply this cryptomorphism to obtain the first examples of $q$-PMDs; in particular we show that $q$-Steiner systems are $q$-PMDs. Secondly, using the $q$-matroid structure of the $q$-Steiner system, we derive new subspace designs. This leads in some cases to designs with parameters not previously known. Interestingly, some of the parameters of the designs we obtain from the putative $q$-Fano plane coincide with those obtained by Braun \textit{et al}~\cite{braun2005systematic}. By characterising {the group of automorphisms} of the designs that we obtained from our $q$-PMD construction, we show that the subspace designs of \cite{braun2005systematic} cannot be derived from the $q$-Fano plane via our construction.

This paper is organised as follows. After some preliminary notions in Section \ref{sec-preliminaries}, we prove in Section \ref{sec-cryptomorphisms} the above-mentioned new cryptomorphism for $q$-matroids. An overview of the different (but equivalent) ways to define $q$-matroids is found at the end of this section. In Section \ref{qPMD} we prove that $q$-Steiner systems are examples of the $q$-analogue of a perfect matroid design. We use this to derive new designs from the $q$-Steiner system, using its $q$-matroid structure and its flats, independent spaces, and circuits. Finally, we characterize the automorphism groups of these new $q$-designs in terms of the automorphisms of $q$-Steiner systems from which they are constructed.

\section{Preliminaries}\label{sec-preliminaries}
In this section, we bring together certain fundamental definitions on lattices, $q$-matroids and subspace designs, respectively.
Throughout the paper, $\mathbb{F}_q$ will denote the finite field of $q$ elements, 
$n$ will be a fixed positive integer and $E$ will denote the
$n$-dimensional vector space $\mathbb{F}_q^n$.

\subsection{Lattices}
Let us first recall preliminaries on lattices. The reader is referred to Stanley \cite{stanley:1997} or Aigner \cite{aigner1979combinatorial} for further details. 

\begin{definition}
	Let $(\mathcal{L}, \leq)$ be a partially ordered set. 
	Let $a,b,v \in \mathcal{L}$.
	We say that $v$ is an {\bf upper bound} of $a$ and $b$ if $a\leq v$ and $b\leq v$ and furthermore, we say
	that $v$ is a {\bf least upper bound} if $v \leq u$ for any $u \in  \mathcal{L} $ that is also an upper bound of $a$ and $b$.  
	If a least upper bound of $a$ and $b$ exists, then it is unique, is denoted by $a\vee b$ and called the {\bf join} of $a$ and $b$.
	We analogously define {\bf a lower bound} and {\bf the greatest lower bound} of $a$ and $b$ and denote the unique greatest lower bound
	of $a$ and $b$ by $a\wedge b$, which is called the {\bf meet} of $a$ and $b$.
	The poset $\mathcal{L}$ is called a {\bf lattice} if each pair of elements has a
	least upper bound and greatest lower bound and denoted by $(\mathcal{L}, \leq, \vee, \wedge)$.
\end{definition}	

Of particular relevance to this paper is the subspace lattice $(\mathcal{L}(E), \leq, \vee, \wedge)$, which is the lattice of $\mathbb{F}_q$- subspaces of $E$, ordered with respect to inclusion and for which the join of a pair of subspaces is their vector space sum and the meet of a pair of subspaces is their intersection. That is, for all subspaces $A,B \subseteq E$ we have:
$$ A \leq B \Leftrightarrow A \subseteq B, A \vee B = A + B, A \wedge B = A \cap B.$$

\begin{definition}
	Let $(\mathcal{L}, \leq, \vee, \wedge)$ be a lattice and let $a,b\in\mathcal{L}$ with $a\leq b$ but $a\neq b$, we say that $b$ {\bf covers} $a$ if for all $c\in \mathcal{L}$ we have that $a\leq c \leq b$ implies that $c=a$ or $c=b$. A {\bf chain} of length $r$ between two elements $a,b\in\mathcal{L}$ is a sequence of distinct elements $a_0,a_1,\ldots,a_r$ in $\mathcal{L}$ such that $a=a_0\leq a_1\leq\cdots\leq a_r=b$. If $a_{i+1}$ covers $a_i$ for all $i$, we call this a {\bf maximal chain}.
\end{definition}

\begin{definition}
	Let $(\mathcal{L}, \leq, \vee, \wedge)$ be a finite lattice. We say that $\mathcal{L}$ is a {\bf semimodular lattice} if it has the property that if $a$ covers $a\wedge b$ then $a\vee b$ covers $b$.
\end{definition}

\begin{definition}
{A lattice $\mathcal{L}$ is called {\bf geometric} if it is
\begin{enumerate}
    \item{atomic (every element is a supremum of the elements covering the unique minimal),}
    \item{semimodular,}
    \item{without infinite chains.}
\end{enumerate}}
\end{definition}

\begin{definition}
	A bijection $\phi: \mathcal{L} \to \mathcal{L}$ on a lattice $(\mathcal{L}, \leq, \vee, \wedge)$ is called an {\bf automorphism} of 
	$\mathcal{L}$ if one of the following equivalent conditions holds for all $a, b \in \mathcal{L}$:
	\begin{enumerate}
		\item{$a \leq b$ iff $\phi(a)\leq \phi(b)$,}
		\item{$\phi(a \vee b)= \phi(a) \vee \phi(b)$,}
		\item{$\phi(a \wedge b)= \phi(a) \wedge \phi(b)$.}
	\end{enumerate}
\end{definition}

\subsection{$q$-Matroids}

{The general framework of defining matroid-like structures over modular complemented lattices is treated in \cite{crapo1964theory}. Important examples of complemented modular lattices are the Boolean lattice, resulting in classical matroids, and the subspace lattice, leading to $q$-matroids.}

For background on the theory of matroids we refer the reader to \cite{gordonmcnulty} or \cite{oxley}.
For the $q$-analogue of a matroid we follow the treatment of Jurrius and Pellikaan~\cite{JP18}. The definition of a $q$-matroid is a straightforward generalisation of the definition of a classical matroid in terms of its rank function. We remark that this definition in fact does not require $E$ to be over a finite field.  
However, as we are focussed on vector spaces over finite fields, we will assume in our definition that a $q$-matroid is an object defined with respect to an $\mathbb{F}_q$-vector space.

\begin{definition}\label{rankfunction}
	A $q$-{\bf matroid} $M$ is a pair $(E,r)$ where
	$r$ is an integer-valued 
	function defined on the subspaces of $E$ with the following properties:
	\begin{description}
		\item[(R1)] For every subspace $A\subseteq E$, $0\leq r(A) \leq \dim A$. 
		\item[(R2)] For all subspaces $A\subseteq B \subseteq E$, $r(A)\leq r(B)$. 
		\item[(R3)] For all $A,B$, $r(A+ B)+r(A\cap B)\leq r(A)+r(B)$.  
	\end{description}
	The function $r$ is called the {\bf rank function} of the $q$-matroid. 
\end{definition}

We list some examples of $q$-matroids \cite{JP18}.

\begin{example}\label{ex:uniform}{[The uniform $q$-matroid]}
	Let 
	$M=(E,r)$, where $$r(A)=\begin{cases}
	\dim A, & \text{ if } \dim A \leq k, \\
	k, & \text{ if } \dim A > k,
	\end{cases}$$
	for $0\leq k \leq n$ and a subspace $A$ of $E$. Then $M$ satisfies axioms (R1)-(R3) and is called the uniform $q$-matroid. We denote it by $U_{k,n}(\mathbb{F}_q)$.
\end{example}

\begin{example}\label{ex:rankmetric}{[Representable $q$-matroid]} 
	Let $G$ be a full-rank $k\times n$ matrix over an extension field $\mathbb{F}_{q^m}$ of $\mathbb{F}_q$. For any subspace $A\subseteq E$
	define the rank of $A$ to be {$r(A) = \mathrm{rank}_{\mathbb{F}_{q^m}}(GY)$}
	for any $\mathbb{F}_q$-matrix $Y$ whose columns span $A$. It can be shown that $(E,r)$ satisfies (R1)-(R3) and hence is a $q$-matroid.
\end{example} 

In classical matroid theory, there are several definitions of a matroid in terms of the axioms of its independent spaces, bases, flats, circuits, etc. 
These equivalences, which are not immediately apparent, are referred to in the literature as {\it cryptomorphisms}. 
In this paper we will establish 
a new cryptomorphism for $q$-matroids. 
First, we define independent spaces, flats, and the closure function in terms of the rank function of a $q$-matroid.

\begin{definition}\label{independentspaces}
	Let $(E,r)$ be a $q$-matroid. 
	A subspace $A$ of $E$ is called {\bf independent} if $$r(A)=\dim A.$$
	We write ${\mathcal I}_r$ to denote the set of independent spaces of the $q$-matroid $(E,r)$.  
	A subspace that is not independent is called {\bf dependent}. 
	We call $C$ a {\bf circuit} if it is itself a dependent space and every proper subspace of $C$ is independent.
\end{definition}

\begin{definition}\label{flat}
{
	Given a $q$-matroid $(E,r)$, a subspace $F \subseteq E$  is called a {\bf flat} } if for all one-dimensional 
	subspaces $x$ such that $x\nsubseteq F$ we have that $$r(F+x)>r(F).$$
	We write ${\mathcal F}_r$ to denote the set of flats of the $q$-matroid $(E,r)$.
\end{definition}

We define the notion of a flat via axioms, without reference to a rank function.

\begin{definition}\label{flat-axioms}
	Let $\mathcal{F} \subseteq \mathcal{L}(E)$. We
	define the following flat axioms:
	\begin{description}
		\item[(F1)] $E\in\mathcal{F}$.
		\item[(F2)] If $F_1\in\mathcal{F}$ and $F_2\in\mathcal{F}$, then $F_1\cap F_2\in\mathcal{F}$.
		\item[(F3)] For all $F\in\mathcal{F}$ and $x\subseteq E$ a one-dimensional subspace not contained in $F$, there is a unique $F'\in\mathcal{F}$ covering $F$ such that $x\subseteq F'$.
	\end{description}
If $\mathcal{F}$ satisfies (F1)-(F3) then we call its members {\bf flats}.
	We write $(E,\mathcal{F})$ to denote a vector space $E$ together with a family of flats satisfying the flat axioms.
\end{definition}

We will see in Section \ref{sec-cryptomorphisms} that a space of flats $(E,\mathcal{F})$ completely determines a $q$-matroid.
The following theorem summarizes important results from \cite{JP18}. 
\begin{theorem}\label{th:JP18prel}
	Let $(E,r)$ be a $q$-matroid and let $A,B\subseteq E$ and let $x,y \subseteq E$ each have dimension one. The following hold.
	\begin{enumerate}
		\item\label{th:JP18prel:part1}{ $r(A+x) \leq r(A)+1$.} 
		\item \label{th:JP18prel:part2} If $r(A+z)=r(A)$ for each one-dimensional space $z\subseteq B$, $z\nsubseteq A$ then $r(A+B)=r(A)$. 
		
		\item\label{th:JP18prel:part3} If $r(A+x)=r(A+y)=r(A)$ then $r(A+x+y)=r(A)$. 
	\end{enumerate}
\end{theorem}  

An interesting family of matroids, the PMDs were introduced in \cite{murty1970equicardinal,Young}. For more details in the classical case, we refer the reader to the work of Deza~\cite{deza1992perfect}. We consider here a $q$-analogue of a PMD.

\begin{definition}
	A $q$-{\bf perfect matroid design} ($q$-PMD) is a $q$-matroid with the property that any two of its flats of the same rank have the same dimension. 
\end{definition}

\subsection{Subspace designs}

Given a pair of nonnegative integers $N$ and $M$, {$M \leq N$,} the $q$-{\bf binomial} or {\bf Gaussian coefficient} counts the number of $M$-dimensional subspaces of an $N$-dimensional subspace over $\mathbb{F}_q$ and is given by:
$$\qbin{N}{M}{q}:=\prod_{i=0}^{M-1}\frac{q^N-q^i}{q^M-q^i}.$$
We write $\qbin{E}{k}{q}$ to denote the set of all $k$-subspaces of $E$ (the $k$-Grassmanian of $E$).

\noindent Recall the following well-known result.
\begin{lemma}\label{lem:folklore}
	Let $s,t$ be positive integers satisfying $0\leq t \leq s \leq n$. The number of $s$-spaces of $E$ that contain a fixed $t$-space is given by $\qbin{n-t}{s-t}{q}$.	
\end{lemma}

We recall briefly the definition of a subspace design and well known examples of these combinatorial objects. The interested reader is referred to the survey \cite{braun2018q} and the references therein for a comprehensive treatment of designs over finite fields. For more recent results, see also \cite{BN19,buratti+}.
\begin{definition}
	Let $1\leq t\leq k \leq n$ be integers and let $\lambda\geq 0$ be an integer. A $t$-$(n,k,\lambda;q)$ {\bf subspace design} is a pair $(E,\mathcal{B})$, where $\mathcal{B}$ is a collection of subspaces of $E$ of dimension $k$, called blocks, with the property that every subspace of $E$ of dimension $t$ is contained in exactly $\lambda$ blocks. 
\end{definition}

Subspace designs are also known as designs over finite fields. A $q$-{\bf Steiner system} is a $t$-$(n,k,1;q)$ subspace design and is said to have parameters
$S(t,k,n;q)$. The $q$-Steiner triple systems are those with parameters $S(2,3,n;q)$ and are denoted by $STS(n;q)$. 
The $t$-$(n,k,\lambda;q)$ subspace designs with $t=1$ and $\lambda =1$ are examples of spreads.

\begin{example}
	A $q$-analogue of the Fano plane would be given by an $STS(7;q)$, whose existence is an open question for any $q$. 
\end{example}

For a subspace $U$ of $E$ we define $U^\perp:=\{ v \in E : \langle u, v \rangle = 0\}$ to be the orthogonal space of $U$ with respect to the scalar product $\langle u, v \rangle=\sum_{i=1}^{n} u_i v_i$.
We will use the notions of the supplementary and dual subspace designs ~\cite{suzuki1990inequalities, KP}.

\begin{definition}\label{lem:dualdesign} Let $k,t,\lambda$ be positive integers and let $\mathcal{D}=(E, \mathcal{B})$ be a $t$-$(n,k,\lambda;q)$ design.
	\begin{enumerate}
		\item The {\bf supplementary design} of $\mathcal{D}$ is the subspace design $\left(E,\qbin{E}{k}{q}- \mathcal{B}\right)$.\\ 
		It has parameters $t$-$\left(n,k,\qbin{n-t}{k-t}{q}-\lambda;q\right)$.
		\item The {\bf dual design} of $\mathcal{D}$ is given by $(E, {\mathcal B}^\perp)$, where $ \mathcal B^\perp:=\{ U^\perp : U \in \mathcal B \}$. It has parameters $$t- \left(n,n-k,\lambda\qbin{n-t}{k}{q}\qbin{n-t}{k-t}{q}^{-1};q \right).$$
	\end{enumerate}
\end{definition} 

The {\it intersection numbers} $\lambda_{i,j}$ defined in Lemma~\ref{intersection_numbers} were given in \cite{KP} and \cite{suzuki1990inequalities}. These design invariants play an important role in establishing non-existence of a design for a given set of parameters.

\begin{lemma}\label{intersection_numbers}
	Let $k,t,\lambda$ be positive integers and let $\mathcal D$ be a $t$-$(n,k,\lambda;q)$ design. 
	Let $I,J$ be $i,j$ dimensional subspaces of $\mathbb{F}_{q}^n$ satisfying $i+j \leq t$ and $I \cap J = \{0\}$.
	Then the number
	$$\lambda_{i,j}:=|\{  U \in \mathcal B : I \subseteq U, \ J \cap U = \{0\} \}|, $$ where $\mathcal{B}$ is the set of blocks of $\mathcal{D}$, 
	depends only on $i$ and $j$, and is given by the formula
	$$ \lambda_{i,j} = q^{j(k-i)}\lambda \qbin{n-i-j}{k-i}{q}\qbin{n-t}{k-t}{q}^{-1}.$$
\end{lemma}  

By Lemma~\ref{intersection_numbers}, the existence of a $t$-$(n,k,\lambda;q)$ design implies the {\it integrality conditions}, namely that $\lambda_i={\lambda}_{i,0}$ are positive integers for $0 \leq i\leq t$.

\begin{definition}
	A parameter set $t$-$(n,k,\lambda; q)$ is called {\bf admissible} if it satisfies the integrality conditions and is called {\bf realisable} if 
	a $t$-$(n,k,\lambda; q)$ design exists. 
\end{definition}

It is well-known and follows directly from the integrality conditions that an $STS(n;q)$ is admissible if and only if $n\equiv 1 \text{ or } 3 \mod{6}$. 
More generally, it was observed in \cite{buratti+} that a $\mathcal{S}(2,k,n;q)$ Steiner system exists only if $n\equiv 1,k \mod k(k-1)$.  

Finally, for a given subspace design $(E,\mathcal{B})$, an automorphism $\phi$ of $\mathcal{L}(E)$ is called an {\bf automorphism of the design} if $\phi(\mathcal{B})= \mathcal{B}$. We will denote the automorphism group of the design $\mathcal{D}=(E,\mathcal{B})$ by $\mathrm{Aut}(\mathcal{D})$ or by
$\mathrm{Aut}(E,\mathcal{B})$.
The automorphism group of a subspace design is equal to the automorphism group of its supplementary design and is in $1-1$ correspondence with that of the dual design.
Automorphism groups have been leveraged to construct new subspace designs using the Kramer-Mesner method~\cite{KramerMesner}. 
If the number of orbits of the automorphism group is small enough, then the corresponding diophantine system of equations can be solved in a feasible amount of time on a personal computer~\cite{braun2005some, braun2005systematic}. It is known that the {binary} $q$-Fano plane has automorphism group of order at most $2$~\cite{BraunKier,KierKurz}, so this method cannot be applied in this case. 

\section{A Cryptomorphism of $q$-Matroids}\label{sec-cryptomorphisms}

In this section we provide a new cryptomorphic definition of a $q$-matroid, in terms of its flats. {The proofs of this cryptomorphism largely follow the classical case . We include the details for expository purposes.} \\
Recall that a flat of a $q$-matroid $(E,r)$ is a subspace $F$ such that for all one-dimensional spaces $x\not\subseteq F$ we have that $r(F+x)>r(F)$. We remark that the results of this section hold for $q$-matroids defined with respect to finite dimensional vector spaces over arbitrary fields.

\begin{definition}\label{Cover}
	Let $F_1$ and $F_2$ be flats of a $q$-matroid.
	We say that $F_1$ {\bf covers} $F_2$ if $F_2\subseteq F_1$ and there is no other flat $F'$ such that $F_2\subseteq F'\subseteq F_1$. 
\end{definition}

Before establishing a cryptomorphism between the $q$-matroids $(E,r)$ and $(E,\mathcal{F})$, we prove some preliminary results.

\begin{lemma}\label{Uguagl}
	Let $(E,r)$ be a $q$-matroid with rank function $r$. Let $A\subseteq B$ be subspaces of $E$ and let $x$ be a one-dimensional subspace of $E$. 
	If $r(B+x)=r(B)+1$ then $r(A+x)=r(A)+1$.
\end{lemma}
\begin{proof}
	Suppose that $r(B+x)=r(B)+1$.	
	Since $A \subseteq B$, we have $(A+x)+B = B+x$ and $A \subseteq (A+x) \cap B$. 	
	Therefore, by (R2) and applying (R3) to $A+x$ and $B$ we get: 
	$$ r(A+x)+r(B) \geq r((A+x)+B)+r((A+x)\cap B) \geq r(B+x)+r(A) =r(B)+1+ r(A), $$
	and so $r(A+x) \geq r(A)+1$.
	By Theorem \ref{th:JP18prel}, $r(A+x) \leq r(A)+1$ and so we get the equality $r(A+x) = r(A)+1$.
\end{proof}

\begin{lemma}\label{F2dasolo}
	If $F_1,F_2$ are two flats of a $q$-matroid $(E,r)$, then $F_1\cap F_2$ is also a flat.
\end{lemma}
\begin{proof}
	Let $F:=F_1\cap F_2$ and take a one-dimensional space $x \nsubseteq F$; therefore $x$ is not a subspace of $F_1$ or $F_2$; say, without loss of generality, that  $x \nsubseteq F_1$. By Theorem \ref{th:JP18prel}, $r(F_1+x)=r(F_1)+1$ and by Lemma \ref{Uguagl}, $r(F+x)=r(F)+1  > r(F)$, which implies that $F$ is flat of $(E,r)$.
\end{proof}

\begin{definition}
	Let $\mathcal{F}$ be a collection of subspaces of $E$ and let
	$A\subseteq E$ be a subspace. We define the subspace
	$$C_{\mathcal{F}}(A):=\bigcap \{F\in \mathcal{F} : A \subseteq F\} .$$
\end{definition}

\begin{lemma}\label{MinimalFlat}
	Let $\mathcal{F}$ be a collection of subspaces of $E$ satisfying the axioms (F1)-(F3). Let $A  \subseteq E$ be a subspace. Then $C_{\mathcal{F}}(A)$ is the unique flat in $\mathcal{F}$ such that the following hold.
	\begin{enumerate}
		\item $A \subseteq C_{\mathcal{F}}(A)$.
		\item If $A \subseteq F\in \mathcal{F}$, then $C_{\mathcal{F}}(A)\subseteq F$.
		\label{MinimalFlatPart3} 
	\end{enumerate}
	 Moreover, if $A\subseteq B \subseteq E$, then  $C_{\mathcal{F}}(A) \subseteq C_{\mathcal{F}}(B)$.
\end{lemma}
\begin{proof}
	(1) and (2) follow immediately from the definition of $C_{\mathcal{F}}(A)$, which is clearly uniquely determined because if there were two flats satisfying these properties, their intersection would violate (2).
	If $B \subseteq F$ for some flat $F$ then $A \subseteq F$ and so clearly, $C_{\mathcal{F}}(A) \subseteq C_{\mathcal{F}}(B)$.
\end{proof}

In the instance that $\mathcal{F}$ is the set of flats of a $q$-matroid $(E,r)$, then from Lemma \ref{F2dasolo}, $C_{\mathcal{F}}(A)$ is itself a flat, which we denote by $F_A$. In particular, $F_A$ is the unique minimal flat of $\mathcal{F}_r$ that contains $A$.

\begin{lemma}\label{minFl}
	Let $(E,r)$ be a $q$-matroid, 	
	let $G$ be a subspace of $E$ and let $x$ be a one-dimensional subspace such that $r(G)=r(G+x)$. 
	Then $x \subseteq F_G$. 
\end{lemma}
\begin{proof}
	Suppose, towards a contradiction, that $x \nsubseteq F_G$. We apply (R3) to $F_G$ and $G+x$:
	$$r(F_G+G+x)+r(F_G \cap (G+x)) \leq r(F_G)+r(G+x).$$
	Now since $G\subseteq F_G$ but $x\not\subseteq F_G$, the above inequality can be stated as
	$$r(F_G+x)+r(G) \leq r(F_G)+r(G).$$
	However, as $F_G$ is a flat, $r(F_G+x)=r(F_G)+1$, which gives the required contradiction.
\end{proof}

\begin{lemma}\label{EqRank}
	Let $(E,r)$ be a $q$-matroid and let $G \subseteq E $. Then $r(G)=r(F_G)$.
\end{lemma}
\begin{proof}
	Consider the collection of subspaces $$\mathcal{H}:=\{y\subseteq E:\, \dim(y)=1, r(G+y)=r(G)\}.$$ Let $U$ be the vector space sum of the elements of $\mathcal{H}$.
	By applying Theorem~\ref{th:JP18prel} Part \ref{th:JP18prel:part2}, we have that $r(U)=r(G)$.
	Moreover, $U\subseteq F_G$ by Lemma \ref{minFl}.
	
	Suppose $r(G)<r(F_G)$. If $U=F_G$ then we would arrive at the contradiction $r(U)=r(F_G)>r(G)$, so assume otherwise. 
	Then there exists a one-dimensional subspace $x \subseteq F_G$, $x\not \subseteq U$. Since $x \notin \mathcal{H}$ and $G\subseteq U$, by (R2) we have
	$$r(U)=r(G) < r(G+x) \leq r(U+x). $$
	On the other hand, Lemma \ref{minFl} tells us that for a one-dimensional subspace $x'\subseteq E$, $x'\not\subseteq F_G$ we have
	$$r(U)=r(G)<r(G+x')\leq r(U+x').$$
	Therefore $U$ is itself a flat and $G \subseteq U \subsetneq F_G$, contradicting the minimality of $F_G$.
	We deduce that $r(G)=r(F_G)$.
\end{proof}

\begin{proposition}\label{flataxioms}
	The flats of a $q$-matroid satisfy the flat axioms (F1)-(F3) of Definition \ref{flat-axioms}.
\end{proposition}
\begin{proof}
	Let $(E,r)$ be a $q$-matroid with rank function $r$.
	By definition, the set of flats ${\mathcal F}_r$ of $(E,r)$ is characterised by:
	$$\mathcal{F}_r:=\{F \subseteq E : r(F+x)>r(F),\, \forall x \nsubseteq F,\, \dim(x)=1\}. $$
	The condition (F1) holds vacuously,
	while (F2) comes from Lemma \ref{F2dasolo}.
	
	To prove (F3), let $F\subseteq E$ and $x \subseteq E$ with $\dim(x)=1$ and $x \nsubseteq F$.
	We will show that there is a unique $F'$ covering $F$ and containing $x$.
	Suppose, towards a contradiction, that $x$ is not contained in any flat covering  $F$. Let $G=F+x$ and consider $F_G$, the minimal flat containing $G$. By our assumption, there must be a flat $F'$ such that $F \subsetneq F' \subsetneq F_G$.
	Without loss of generality, we may assume that $F'$ is a cover of $F$. Clearly $x \nsubseteq F'$. 
	Let $y$ be  a one-dimensional space $y\subseteq F'$, $y \nsubseteq F$. 
	Now, $x,y\nsubseteq F$ and $y\subseteq F_G$. 
	Let $H=F+y$. We claim that $x \subseteq F_H$, in which case we would arrive at the contradiction $x \subseteq F_H \subseteq F'$ and $x \nsubseteq F'$. 
	Since $G=F+x$, $H=F+y$,  $x,y\nsubseteq F$ and $F$ is a flat, we have $r(G)=r(H)=r(F)+1$. 
	By Lemma \ref{EqRank}, $r(G)=r(F_G)$ and since $y\subseteq F_G$ we also have $r(G)=r(G+y)=r(F_G)$.
	Now, $$r(H+x)=r(F+x+y)=r(G+y)=r(G)=r(F)+1=r(H).$$ Hence by Lemma \ref{minFl}, $x \subseteq F_H$.
	We deduce that $x$ is contained in a cover of $F$.
	As regards uniqueness, 
	suppose we have two different covers $F_1 \neq F_2$ of $F$ containing $x$ and let $L:=F_1\cap F_2$.
	By the flat axiom (F2), $L$ is a flat and since $x,F \subseteq F_1,F_2$ then $x,F \subseteq L$. On the other hand, $F \neq L$ since $x \nsubseteq F$, so $F \subsetneq L$.
	Since $F_1 \neq F_2$, $L$ cannot be equal to both of them; say $L \neq F_2$, so $F \subsetneq L \subsetneq F_2$, which contradicts the fact that $F_2$ covers $L$.
\end{proof}

Our aim is to prove the converse of Proposition \ref{flataxioms}: that is, if we have a collection of flats $\mathcal{F}$ that satisfies the axioms (F1)-(F3), it is the collection of flats of a $q$-matroid. The next lemma will be used frequently in our proofs.

\begin{lemma}\label{cover}
	Let $\mathcal{F}$ be a collection of flats. Let $F\in\mathcal{F}$ and let $x\subseteq E$ be a one-dimensional subspace. Then the minimal member of $\mathcal{F}$ containing $F+x$ is either equal to $F$ or it covers $F$.
\end{lemma}
\begin{proof}
	If $x\subseteq F$, then $F+x=F$ so the minimal member of $\mathcal{F}$ containing $F+x$ is $F$ itself. If $x\not\subseteq F$, then by (F3) there is a unique $F'\in\mathcal{F}$ that covers $F$ and contains $x$. Since $F'$ covers $F$ and contains both $F$ and $x$, it is clearly the minimal member of $\mathcal{F}$ containing $F+x$.
\end{proof}

Next we show that the members of $\mathcal{F}$ form a semimodular lattice. (The flats of a $q$-matroid form in fact a geometric lattice, as was noted in Theorem 1 of \cite{BCJ17}.)

\begin{theorem}\label{semimodularLat}
    Let $\mathcal{F}$ be a collection of flats.
	Then its members form a semimodular lattice under inclusion, where for any two $F_1,F_2\in\mathcal{F}$ the meet is defined to be $F_1\wedge F_2 :=F_1\cap F_2$ and the join $F_1\vee F_2$ is $C_\mathcal{F}(F_1+F_2)$.
\end{theorem}
\begin{proof}
	The members of $\mathcal{F}$ clearly form a poset with respect to inclusion. We prove that the definitions of meet and join as 
	$F_1\wedge F_2:=F_1\cap F_2$ and 
	$F_1\vee F_2 := C_\mathcal{F}(F_1+F_2)$ are well defined.
	
	Let us consider the meet.
	From (F2) $F_1 \wedge F_2$ is in $\mathcal{F}$ and the fact that it is the greatest lower bound of $F_1$ and $F_2$ follows from the definition of intersection. As regards the join, $C_\mathcal{F}(F_1+F_2)$ is in $\mathcal{F}$ by Lemma \ref{MinimalFlat} and, more precisely, is the unique minimal member of $\mathcal{F}$ containing $F_1+F_2$.
	We remark that, since we have a lattice, there is a maximal member of $\mathcal{F}$, which is $E$, and a minimal one, that is $\cap\{F \in \mathcal{F}\}$, which is also the minimal member of $\mathcal{F}$ containing the zero space. 
	
	In order to prove that the lattice is semimodular, we have to prove that if $F_1\wedge F_2$ is covered by $F_1$, then $F_2$ is covered by $F_1\vee F_2$. So let $F_1\cap F_2\in\mathcal{F}$ be covered by $F_1$. Then for all 
	one-dimensional subspaces $x\subseteq F_1$ but $x\nsubseteq F_2$ we have that the minimal member of $\mathcal{F}$ containing $(F_1\cap F_2)+ x$ is $F_1$ by Lemma \ref{cover}. Since $F_2+x\subseteq F_2+F_1$, we have that the minimal $H\in\mathcal{F}$ containing $F_2+ x$ satisfies $H \leq F_2\vee F_1$. 
	On the other hand, because $(F_1\cap F_2)+x\subseteq F_2+x$, we have that 
	$F_1 \leq H$. Now we have that both 
	$F_1, F_2 \leq H$ so $H$ must contain the least upper bound of the two, that is, $H\geq F_1\vee F_2$. We conclude that $H=F_1\vee F_2$, which means $F_1\vee F_2$ covers $F_2$ by Lemma \ref{cover}. This proves that the lattice of a collection of flats $\mathcal{F}$ is semimodular.
\end{proof}

Since the lattice of a collection of flats is semimodular, we can deduce the following corollary (see \cite[Prop.~3.3.2]{stanley:1997}, \cite[Prop.~3.7]{stanley:2007} or \cite[Prop.~2.1]{aigner1979combinatorial}.)

\begin{corollary}\label{JordanDedekind}
	The lattice of a collection of flats $\mathcal{F}$ satisfies the Jordan-Dedekind property, that is: all maximal chains between two fixed elements of the lattice have the same finite length.
\end{corollary}

In what follows, we will need the following lemma.
\begin{lemma}\label{GettingFA}
	Let $A$ be a subspace of $E$ and let $\mathcal{F}$ be a collection of subspaces of $E$. 
	Let $x \subseteq A$ have dimension one and let $F\subseteq A$ be an element of $\mathcal{F}$.
	Let $F'$ be the minimal element of $\mathcal{F}$ containing $x+F$.
	If $A \subseteq F'$ then $F'=C_{\mathcal {F}}(A)$.
\end{lemma}
\begin{proof}
	If $A \subseteq F' \in \mathcal{F}$, we have $C_{\mathcal {F}}(A) \subseteq F'$ by definition.
	Then since $F+x \subseteq A$ we have $F+x \subseteq C_{\mathcal {F}}(A) \subseteq F'$.
	Since $F'$ is the the minimal flat containing $
	F$ and $x$, $F' \subseteq C_{\mathcal {F}}(A)$, implying their equality.
\end{proof}

For each $A\subseteq E$, let $r_{\mathcal{F}}(A)$ denote the length of a maximal chain of flats from $C_{\mathcal {F}}(\{0\})$ to $C_{\mathcal {F}}(A)$. By Corollary~\ref{JordanDedekind}, all such maximal chains have the same length, so $r_{\mathcal{F}}$ is well defined as a function on $\mathcal{L}(E)$. We are now ready to prove our main theorem. 

\begin{theorem}\label{crypto:flat-rank}
	Let $E$ be a finite dimensional space. If $\mathcal{F}$ is a family of subspaces of $E$ that satisfies the flat axioms (F1)-(F3)
	and for each $A \subseteq E$, then $(E,r_\mathcal{F})$ is a $q$-matroid and its family of flats is $\mathcal{F}$. 
	Conversely, for a given $q$-matroid $(E,r)$, $\mathcal{F}_r$ satisfies the conditions (F1)-(F3) and $r=r_{\mathcal{F}_r}$.
\end{theorem}
\begin{proof}
	Let $(E,r)$ be a $q$-matroid.
	We have seen in Proposition \ref{flataxioms} that $\mathcal{F}_r$ satisfies (F1)-(F3).
	
	Let now  $(E,\mathcal{F})$ be a family of flats.
	Write $F_0$ to denote $C_{\mathcal {F}}(\{0\})$. We will show that $r_{\mathcal{F}}$ satisfies (R1)-(R3), that is, that $(E,r_{\mathcal{F}})$ is a $q$-matroid.
	
	(R1): For a subspace $A$, $r_{\mathcal{F}}(A)\geq 0$ since $C_{\mathcal {F}}(A)$ is contained in any chain from $F_0$ to $C_{\mathcal {F}}(A)$. 
	If $A \subseteq F_0$ then $F_0=C_{\mathcal {F}}(A)$ and $r_{\mathcal{F}}(A)=0 \leq \dim (A)$, so the result clearly holds. 
	If $F_0$ does not contain $A$, then there is a one-dimensional space $x_0 \subseteq  A$, $x_0 \not\subseteq F_0$. 
	Let $G_0=F_0+x_0$ and define $F_1$ to be the minimal flat containing $F_0$ and $x_0$. $F_1$ clearly has dimension at least 1 and is a cover of $F_0$. Indeed if there is a flat $H$ such that $F_0 \subsetneq H \subsetneq F_1$, $H$ contains $F_0$ properly (otherwise $H=F_0$) and $x_0 \nsubseteq H$ (otherwise $H=F_1$). If it contains any element of $F_0+x_0$ that is not in $F_0$ it would contain $x_0$ itself as a subspace.
	If $A \subseteq F_1$, then by Lemma \ref{GettingFA} we have $F_1=C_{\mathcal {F}}(A)$ and the required maximal chain is $F_0 \subsetneq C_{\mathcal {F}}(A)$.
	If $A \nsubseteq F_1$ then choose $x_1 \subseteq A$ but $x_1 \nsubseteq F_1$ and define $F_2$ to be the unique cover of $G_1=F_1+x_1$; clearly $\dim(F_2) \geq 2$.
	We continue in this way, choosing at each step a one-dimensional subspace $x_i \subseteq A, x_i \not\subseteq F_{i}$ and construct the unique flat $F_{i+1}$
	covering $G_{i} = F_{i}+x_i$, until we arrive at a flat $F_k$ that contains $A$.
	By Lemma \ref{GettingFA}, we have $F_k=F_A$, yielding the maximal chain of flats $F_0 \subseteq F_1 \subseteq \cdots \subseteq F_k = C_{\mathcal {F}}(A)$. 
	Since $\dim(F_i)\geq i$ for each $i$, it follows that $r_{\mathcal{F}}(A) = k \leq \dim(A)$.
	
	(R2): Let $A \subseteq B$; we prove $r_{\mathcal{F}}(A) \leq r_{\mathcal{F}}(B)$.
	By Lemma \ref{MinimalFlat} Part \ref{MinimalFlatPart3}, $C_{\mathcal {F}}(A) \subseteq C_{\mathcal {F}}(B)$, therefore a maximal chain of flats $F_0\subseteq F_1\subseteq\cdots\subseteq C_{\mathcal {F}}(A)$ is contained in $C_{\mathcal {F}}(B)$.
	
	(R3): Let $A, B \subseteq E$ be subspaces and consider a maximal chain of flats $F_0 \subseteq \cdots \subseteq C_{\mathcal {F}}(A\cap B)$.
	If $C_{\mathcal {F}}(A\cap B)\neq C_{\mathcal {F}}(A)$ then choose a one-dimensional space $x_1\subseteq A, x_1 \not \subseteq C_{\mathcal {F}}(A\cap B)$ and continue extending the chain, by setting $G_1=C_{\mathcal {F}}(A\cap B)+x_1$ and taking $F_{1}=C_{\mathcal {F}}(G_1)$ and then
	repeating this procedure, each time choosing $x_i \subseteq A, 
	x_i\not\subseteq F_{i-1}$,
	where $F_i$ is the cover of
	$G_i=F_{i-1}+x_{i}$
	for each $i$.
	This sequence is clearly finite (in fact has length at most 
	$\dim(A)-\dim(C_{\mathcal{F}}(A\cap B))$,
	and by Lemma \ref{GettingFA}, there exists some $k$ such that $F_k = C_{\mathcal {F}}(A)$.
	Once we have a maximal chain terminating at $C_{\mathcal {F}}(A)$, if $B\nsubseteq C_{\mathcal {F}}(A)$, we repeat the same procedure, constructing a maximal chain terminating at 
	$C_{\mathcal {F}}(A+B)$.
	In the same way, from $F_0 \subseteq ... \subseteq C_{\mathcal {F}}(A\cap B)$, we construct a maximal chain terminating at $C_{\mathcal {F}}(B)$, which can be extended to a maximal chain terminating
	at $C_{\mathcal {F}}(A+ B)$.
	For any $y \subseteq C_{\mathcal {F}}(B)$, by (F2), the minimal flat containing $C_{\mathcal {F}}(A\cap B)+y$ is contained in the minimal flat containing $C_{\mathcal {F}}(A)+y$. Repeating this procedure gives us that every flat in the chain from $C_{\mathcal {F}}(A\cap B) $ to $C_{\mathcal {F}}(B)$ is contained in exactly one flat in the chain from $C_{\mathcal {F}}(A)$ to $C_{\mathcal {F}}(A+ B)$, and any flat in the latter chain contains at least one flat of the former chain. In other words, there is a surjection between flats in the chain from $C_{\mathcal {F}}(A\cap B) $ to $C_{\mathcal {F}}(B)$ and the flats in the chain from $C_{\mathcal {F}}(A)$ to $C_{\mathcal {F}}(A+B)$.
Therefore, the length of a maximal chain from $C_{\mathcal {F}}(A\cap B) $ to $C_{\mathcal {F}}(B)$ is longer than or equal to the length of a maximal chain from $C_{\mathcal {F}}(A)$ to $C_{\mathcal {F}}(A+ B)$. This yields
	$$r_{\mathcal{F}}(A+B)-r_{\mathcal{F}}(A) \leq r_{\mathcal{F}}(B)-r_{\mathcal{F}}(A\cap B),$$ and this proves (R3).
	
	The only thing that remains to be proved is that rank and flats defined as above compose correctly, namely $\mathcal{F}\to r\to\mathcal{F}'$ implies $\mathcal{F}=\mathcal{F}'$, and $r\to\mathcal{F}\to r'$ implies $r=r'$.
	Given a family $\mathcal{F}$ of flats satisfying (F1)-(F3), define $r(A)$ to be the length of a maximal chain $F_0\subseteq\cdots\subseteq C_{\mathcal {F}}(A)$. Then let $\mathcal{F}'=
	\mathcal{F}_r=\{F\subseteq E:r(F+x)>r(F), \forall x\not\subseteq F ,\dim(x)=1\}$. We want to show that $\mathcal{F}=\mathcal{F}'$. Let $F\in\mathcal{F}$, which means that $F=C_{\mathcal {F}}(F)$ is the endpoint of a maximal chain. Equivalently, for any one-dimensional subspace $x \subseteq E , x \nsubseteq F$, we have that a maximal chain for $F+x$ has to terminate at a flat that
	properly contains $F$ and so $r(F+x)>r(F)$. Thus $F\in\mathcal{F}'$. 
	
Conversely, if $r$ is a rank function satisfying (R1)-(R3), let $\mathcal{F}=\mathcal{F}_r$. Then let $r_\mathcal{F}(A)$ be the length of a maximal chain $F_0\subseteq\cdots\subseteq C_{\mathcal {F}}(A)$. We want to show that $r=r_\mathcal{F}$. This follows from the same reasoning as above: each element $F \in \mathcal{F}$ is the endpoint of a maximal chain and hence is strictly contained in the unique cover of $x+F$ for any $x \not\subseteq F$.
\end{proof}

\section{$q$-PMD's and Subspace Designs}\label{qPMD}

As an application of the cryptomorphism between the rank and flat axioms proved in Section~\ref{sec-cryptomorphisms}, we obtain the first example of $q$-PMD that has a classical analogue, namely the $q$-Steiner systems. Furthermore, we generalize a result of Murty \textit{et al.}~\cite{Young} and show that from the flats, independent subspaces and circuits of our $q$-PMD we derive subspace designs.
While only the $STS(13,2)$ parameters are known to be realisable to date, we have used it in our construction to obtain subspace designs for parameters that were not previously known to be realisable. Finally, we focus on the automorphism groups of the subspace designs that are considered in this section.

\subsection{$q$-Steiner Systems are $q$-PMDs}\label{q-steiner-q-PMD}      

We start by showing that a $q$-Steiner system gives a $q$-matroid, and we classify its family of flats.  The construction given here uses the flat axioms, whereas in \cite{Young} the hyperplane axioms are used.

\begin{proposition}\label{prop:classifysteinerflats}
	Let $\mathcal{S}$ be a $q$-Steiner system and let $\mathcal{B}$ denote its blocks. We define the family 
	$$\mathcal{F}=\left\{ \bigcap_ { B \in S}B  : S\subseteq \mathcal{B}\right\}.$$  
	Let $F$ be a subspace of $E$. 
	Then $F \in \mathcal{F}$ if and only if one of the following holds:  
	\begin{enumerate}
		\item $F=E$,
		\item $F \in \mathcal{B}$,
		\item $\dim(F)\leq t-1$.
	\end{enumerate}
\end{proposition}
\begin{proof}
	By definition, $\mathcal{F}$ is the collection of all intersections of the blocks in $\mathcal{B}$, so clearly $E \in \mathcal{F}$ (taking the empty intersection) and every block is contained in $\mathcal{F}$. Let us consider the case for which $F$ is the intersection of at least two blocks. Clearly, $\dim(F) \leq t-1$, because every $t$-space is in precisely one block by the Steiner subspace design property. 
	Let $\dim(F)=i$. By Lemma \ref{intersection_numbers}, there are exactly $\lambda_i = \qbin{n-i}{k-i}{q}\qbin{n-t}{k-t}{q}^{-1}$ blocks that contain $F$, so that $F$ is contained in the intersection $I$ of these blocks.
	We claim that $F=I$. If not, then there exists a 1-dimensional space $x \subseteq I$, $x \nsubseteq F$ such that the $(i+1)$-dimensional space $x+F$ is contained in $I$ and so, in particular, $x+F$ is contained in some $\lambda_i$ blocks. However, again by Lemma \ref{intersection_numbers} there are exactly $\lambda_{i+1} = \qbin{n-i-1}{k-i-1}{q}\qbin{n-t}{k-t}{q}^{-1}$ blocks that contain $x+F$, which leads to a contradiction since $\lambda_{i+1} < \lambda_i$. 
	We conclude that every space of dimension at most $t-1$ is contained in $\mathcal{F}$.
\end{proof}

\begin{theorem}\label{thm:steinertoflats}
	Let $\mathcal{S}$ be
	a $q$-Steiner system and let $\mathcal{B}$ denote its blocks. Let $\mathcal{F}$ be defined as in
	Proposition \ref{prop:classifysteinerflats}.
	Then $(E,\mathcal{F})$ defines a $q$-matroid given by its flats.
\end{theorem}
\begin{proof} By the cryptomorphic definition of a $q$-matroid in Theorem \ref{crypto:flat-rank}, it would be enough to show that $\mathcal{F}$ satisfies the axioms (F1), (F2) and (F3). We have that (F1) holds by Proposition \ref{prop:classifysteinerflats}. By the definition of $\mathcal{F}$, we see that (F2) also holds.
	To prove $(F3)$, let $F \in \mathcal{F}$ and let $x \subseteq E$ be a one-dimensional subspace such that $x \nsubseteq F$. 
	If $\dim(F)=k$ then the unique cover of $F$ in ${\mathcal F}$ that contains $x$ is the whole space $E\in \mathcal F$, since no block contains a $(k+1)$-dimensional space.
	Now suppose that $\dim(F) =t-1$. Then $\dim(x+F)=t$ so that there exists a unique block, which is contained in $\mathcal{F}$, that covers $F$ and contains $x$.
	Finally, suppose that $\dim(F) \leq t-2$. Then $\dim(F+x) \leq t-1$, so that by Proposition \ref{prop:classifysteinerflats} $x+F \in \mathcal{F}$, which is clearly the unique cover of $F$ that contains $x$.
\end{proof}

The $q$-matroid $(E,\mathcal{F})$ determined by a $q$-Steiner system as described in Theorem \ref{thm:steinertoflats} is referred to
as the $q$-matroid {\it induced} by the $q$-Steiner system.

In Theorem \ref{crypto:flat-rank} it was shown that a collection of subspaces $\mathcal{F}$ of $E$ satisfying (F1)-(F3) determines a $q$-matroid $(E,r)$ such that for each $A \subseteq E$, $r(A)+1$ is the length of a maximal chain of flats contained in $F_A$. We will now determine explicit values of the rank function of the $q$-matroid induced by a $q$-Steiner system as described in Theorem 
	\ref{thm:steinertoflats}.

\begin{proposition}\label{rank-func-pmd}
	Let $M=(E,\mathcal{F})$ be the $q$-matroid for which $\mathcal{F}$ is the set of
	intersections of the blocks of an $S(t,k,n; q)$ Steiner system 
$ (E,\mathcal{B})$. 
	Then $M$ is a $q$-PMD with rank function defined by
	$$r(A)=\left\{ \begin{array}{ll}
	\dim(A) & \text{ if } \dim(A) \leq t,\\
	t   & \text{ if } \dim(A) > t \textrm{ and } A \text{ is contained in a block of }\mathcal{B}, \\
	t+1  & \text{ if } \dim(A) > t \textrm{ and } A \text{ is not contained in a block of }\mathcal{B}.
	\end{array}
	\right .$$
\end{proposition}

\begin{proof}
	Let $A \subseteq E$ be a subspace. Then $r(A)+1$ is the length of a maximal chain of flats contained in $F_A=C_{\mathcal{F}}(A)$.
	If $\dim A \leq t-1$ then $A$ is a flat, as are all its subspaces. So a maximal chain of flats contained in $F_A=A$ has length $\dim A+1$, hence $r(A)=\dim A$. If $\dim A=t$ then $A$ is contained in a unique block, and this block is equal to $F_A$. A maximal chain of flats
	has the form $F_0\subseteq F_1\subseteq\cdots\subseteq F_{t-1}\subseteq F_A$, where $\dim F_i=i$. This chain has length $t+1$ hence $r(A)=t=\dim A$.
	
	If $\dim A>t$ and $A$ is contained in a block, then $F_A$ is a block and we apply the same reasoning as before to find $r(A)=t$. If $\dim A>t$ and $A$ is not contained in a block, then $F_A=E$ and a maximal chain of flats is $F_0\subseteq F_1\subseteq\cdots\subseteq F_{t-1}\subseteq B\subseteq E$ where $B$ is a block.
	This gives $r(A)=t+1$.
	
	To establish the $q$-PMD property, we show that flats of the same rank have the same dimension.
	Clearly this property is satisfied by flats of dimension at most $t$.
	Let $F$ be a flat of dimension at least $t+1$ and rank $t$. Then $F$ is contained in a unique block and hence, being an intersection of blocks by definition, is itself a block and has dimension $k$. If $F$ has rank $t+1$, then it is not contained in a block, and hence must be $E$.
\end{proof}

\subsection{Subspace Designs from $q$-PMD's} Let $M$ be a $q$-matroid induced by a $q$-Steiner system. We will now give a classification of its flats, independent subspaces and circuits and show that these yield new subspace designs by the idea given in \cite{Young} for the classical case. {However, while the constructions are a direct generalisation to the $q$-analogue, the counting arguments for the parameters in these constructions are considerably more involved than in the classical case.}

\subsection*{Flats}

We have classified the flats of a $q$-matroid induced by a $q$-Steiner system in Proposition \ref{prop:classifysteinerflats}. By considering all flats of a given rank, we thus get the following designs:
\begin{enumerate}
	\item For rank $t+1$ we have only one block, $\mathbb{F}_q^n$. This is an $n$-$(n,n,1)$ design.
	\item For rank $t$, we get the original $q$-Steiner system.
	\item For rank less than $t$ we get a trivial design.
\end{enumerate}

\subsection*{Independent spaces}

\begin{proposition}\label{class_independent}
	Let $M$ be the $q$-PMD induced by a $q$-Steiner system with blocks $\mathcal{B}$. Let $I$ be a subspace of $E$. Then $I$ is independent if:
	\begin{enumerate}
		\item $\dim I\leq t$,
		\item $\dim I=t+1$ and $I$ is not contained in a block of $\mathcal{B}$.
	\end{enumerate}
\end{proposition}
\begin{proof}
	This follows directly from the fact that $I$ is independent if and only if $r(I)=\dim I$ and the definition of the rank function of $M$ in Proposition \ref{rank-func-pmd}.
\end{proof}

We want to know if all independent spaces of a fixed dimension $\ell$ of a given $q$-PMD form the blocks of a $q$-design. There are two trivial cases:
\begin{enumerate}
	\item If $\ell\leq t$ then the blocks are all spaces of dimension $\ell$. This is a trivial design.
	\item If $\ell>t+1$ then there are no independent spaces. This is the empty design.
\end{enumerate}
So, the only interesting case to study is that of the independent spaces of dimension $t+1$. These comprise the $(t+1)$-spaces none of which is contained in a block of $\mathcal{B}$.

\begin{theorem}\label{design_independent}
	Let $M$ be the $q$-PMD induced by a $q$-Steiner system with parameters $S(t,k,n;q)$ and blocks $\mathcal{B}$.
	The independent spaces of dimension $t+1$ of $M$ form a $t$-$(n,t+1,\lambda_{\mathcal{I}})$ design with $\lambda_{\mathcal{I}}=(q^{n-t}-q^{k-t})/(q-1)$.
\end{theorem}
\begin{proof}
	Let $\mathcal{I}$ be the set of independent spaces of dimension $t+1$ of $M$.
	We claim that for a given $t$-space $A$, the number of blocks $I\in\mathcal{I}$ containing it is independent of the choice of $A$, and thus equal to $\lambda_{\mathcal{I}}$.
	
	Let $A$ be a $t$-space and let $\lambda(A)$ denote the number of $(t+1)$-spaces of ${\mathcal I}$ that contain $A$. 
	$A$ is contained in a unique block $B\in\mathcal{B}$ of the $q$-Steiner system. We extend $A$ to a $(t+1)$-space $I$ that is not contained in any block of $\mathcal{B}$, that is, we extend $A$ to $I \in \mathcal{I}$. We do this by taking a $1$-dimensional vector space $x$ not in $B$ and letting $I=A+x$. The number of $1$-spaces not in $B$ is equal to the total number of $1$-spaces minus the number of  one-dimensional spaces in $B$:
	\[\qbin{n}{1}{q}-\qbin{k}{1}{q}=\frac{q^n-1}{q-1}-\frac{q^k-1}{q-1}=q^k
	\qbin{n-k}{1}{q}.\]
	However, another one-dimensional space $y$ that is in $I$ but not in $A$ gives that $A+x=A+y$. The number of one-dimensional spaces in $I$ but not in $A$ is equal to
	\[ \qbin{t+1}{1}{q}-\qbin{t}{1}{q}=\frac{q^{t+1}-1}{q-1}-\frac{q^t-1}{q-1}=q^t. \]
	This means that the number of ways we can extend $A$ to $I \in \mathcal{I}$ is the quotient of the two values calculated above:
	\[\lambda(A)=
	q^{k-t}\qbin{n-k}{1}{q}=
	\frac{q^{n-t}-q^{k-t}}{q-1}=\lambda_{\mathcal{I}}, \]
	which is independent of the choice of $A$ of dimension $t$. 
\end{proof}

\begin{example}
	For the $q$-PMD arising from the 
	$S(2,3,13;2)$ Steiner system we have
	$b_{\mathcal{I}}= 3267963270$
	and 
	$\lambda_{\mathcal{I}} = 2046.$
	For the $q$-PMD arising from the  putative $q$-Fano plane $STS(7;2)$, we have
	$b_{\mathcal{I}} = 11430$ and $\lambda_{\mathcal{I}}=30.$
\end{example}

\begin{remark}
{For $k=t+1$ the construction from Theorem \ref{design_independent} gives the supplementary design of the $q$-Steiner system.}
\end{remark}

\subsection*{Circuits}

\begin{proposition}\label{prop:circuits}
	Let $M$ be a $q$-PMD induced by a $q$-Steiner system $S(t,k,n;q)$ with blocks $\mathcal{B}$. Let $C$ be a subspace of $M$. Then $C$ is a circuit if and only if:
	\begin{enumerate}
		\item $\dim C=t+1$ and $C$ is contained in a block of $\mathcal{B}$,
		\item $\dim C=t+2$ and all $(t+1)$-subspaces of $C$ are contained in none of the blocks of $\mathcal{B}$.
	\end{enumerate}
\end{proposition}
\begin{proof}
	A circuit is a space such that all its codimension $1$ subspaces are independent. All spaces of dimension at most $t$ are independent, so a circuit will have dimension at least $t+1$. Also, since the rank of $M$ is $t+1$, a circuit has dimension at most $t+2$. The result now follows from the definition of a circuit and the above Proposition \ref{class_independent} that classifies the independent spaces of $M$.
\end{proof}

We now show that all the $(t+1)$-circuits form a design and that all the $(t+2)$-circuits form a design.

\begin{theorem}\label{design_circuits}
	Let $M$ be a $q$-PMD induced by a $q$-Steiner system $S(t,k,n; q)$ with blocks $\mathcal{B}$. Let $\mathcal{C}_{t+1}$ be the collection of all circuits of $M$ of dimension $(t+1)$. Then $\mathcal{C}_{t+1}$ are the blocks of a $t$-$(n,t+1,\lambda_{\mathcal{C}_{t+1}})$ design where \[ \lambda_{\mathcal{C}_{t+1}}=\qbin{k-t}{1}{q}.\]
\end{theorem}

\begin{proof}
	Let $A$ be a $t$-space contained in a unique block $B_A$ in the $q$-Steiner system. There are $\qbin{k-t}{t+1-t}{q} = \qbin{k-t}{1}{q}$ $(t+1)$-dimensional subspaces of $B_A$ that contain $A$, from Lemma \ref{lem:folklore}. Every such $(t+1)$-space is a circuit by definition. If $C$ is a circuit not contained in $B_A$, then by Proposition \ref{prop:circuits} $C$ is contained in another block $B \in {\mathcal B}$. Therefore, if $A \subseteq C$, then $A$ is contained in two distinct blocks $B_A$ and $B$, contradicting the Steiner system property. 
	Hence the number of blocks that contain $A$ is$\lambda(A)=\qbin{k-t}{1}{q},$ which is independent of our choice of $A$ of dimension $t$.
\end{proof}

\begin{remark}
	In fact by Proposition \ref{prop:circuits}, Theorem \ref{design_circuits} and Theorem \ref{design_independent} are equivalent. The circuits of dimension $t+1$ are precisely the set $(t+1)$-spaces each of which is contained in some block of the $q$-Steiner system. Therefore this set of circuits is the complement of the set of $(t+1)$- spaces for which none of its members is contained in a block of the Steiner system. It follows that the $q$-designs of Theorems \ref{design_independent} and \ref{design_circuits} are supplementary designs with respect to each other.
\end{remark}

\begin{theorem}\label{design_circuits-2}
	Let $M$ be a $q$-PMD induced by a $q$-Steiner system $S(t,k,n; q)$ with blocks $\mathcal{B}$. Let $\mathcal{C}_{t+2}$ be the collection of all circuits of $M$ of dimension $(t+2)$. Then $\mathcal{C}_{t+2}$ are the blocks of a $t$-$(n,t+2,\lambda_{\mathcal{C}_{t+2}})$ design where
	\[ \lambda_{\mathcal{C}_{t+2}}=q^{k-t}\qbin{n-k}{1}{q}\left( \qbin{n-t-1}{1}{q}-\qbin{k-t}{1}{q}\qbin{t+1}{1}{q}\right)\frac{1}{q+1}. \]
\end{theorem}
\begin{proof}
	Let $\mathcal{C}_{t+2}$ be the set of circuits of dimension $t+2$ of $M$. We argue that every $t$-space is contained in the same number $\lambda_{\mathcal{C}_{t+2}}$ of members of $\mathcal{C}_{t+2}$. We do this by calculating for a given $t$-space $A$ the number of blocks $C\in\mathcal{C}_{t+2}$ it is contained in. It turns out this number is independent of the choice of $A$, and thus equal to $\lambda_{\mathcal{C}_{t+2}}$. 
	
	Define
	$$N(A):=|\{ (I,C): A\subseteq I\subseteq C, I \in {\mathcal I}, \dim I = t+1, C\in  {\mathcal{C}_{t+2}} \}|. $$
	The number of $(t+1)$-dimensional independent spaces $I$ containing $A$ is exactly the number $\lambda_{\mathcal{I}}$ calculated in Theorem \ref{design_independent}, which is $(q^{n-t}-q^{k-t})/(q-1)$. 
	Now let $I$ be an independent space of dimension $t+1$ that contains the $t$-space $A$. Then $I$ is a $(t+1)$-space that is not contained in a block of $\mathcal{B}$. 
	We will count the number of $(t+2)$-dimensional spaces $C$ such that
	$I \subseteq C \in {\mathcal{C}_{t+2}}$. 
	Such a subspace $C$ contains $I$ as a subspace of codimension $1$ and meets any block of $\mathcal{B}$ in a space of dimension at most $t$, by Proposition \ref{prop:circuits}. 
	Define $\mathcal{B}_I:=\{ B \in \mathcal{B} : \dim(B \cap I) = t \}$. 
	Clearly,
	the complement of $\mathcal{C}_{t+2}$ in the set of all $(t+2)$-dimensional spaces containing $I$ is the set of $(t+2)$-dimensional subspaces of
$E$ that contain $I$ and meet some block of $\mathcal{B}_I$ in a $(t+1)$-dimensional space.
	 
	Now fix some $B \in \mathcal{B}_I$. 
	Let 
	$$\mathcal{S}(B,I):=\{ D \subseteq E : \dim(D)=t+2, I \subseteq D, \dim(B \cap D)=t+1 \}$$ and let $$\mathcal{T}(B,I) :=\{ X \subseteq B: \dim(X) = t+1,I \cap B \subseteq X\}.$$
	We now claim that the following are well defined mutually inverse bijections:
	\[ 
	\varphi: \mathcal{S}(B,I) \longrightarrow \mathcal{T}(B,I) : D \mapsto D \cap B,\:
	\phi: \mathcal{T}(B,I)  \longrightarrow \mathcal{S}(B,I) : X \mapsto X+I .
	\]
	Let $D \in \mathcal{S}(B,I)$ and let $X = D \cap B$. Then $I \cap B \subseteq X$ as $I \subseteq D$ and clearly $\dim(X) = t+1$. Therefore $X \in \mathcal{T}(B,I)$ and $\varphi$ is well-defined.
	Conversely, let $X \in \mathcal{T}(B,I)$ and define $D= X+I$. 
	Note first that as $\dim(I \cap B)=t$, $I \cap X = I \cap B$ has codimension 1 in $X$. 
	We have
	$\dim(D) = \dim(X+I)=\dim(X) +\dim(I)-\dim(X \cap I) = t+1 +t+1 - t = t+2$. Clearly, $I \subseteq D$ and
	$\dim(D \cap B) = \dim(D)+\dim(B) - \dim(D + B)=t+2 + k -\dim(I + B) =t+2+k-k-1=t+1$.
	Therefore, $\phi$ is well-defined.
	Let $X \in \mathcal{T}(B,I)$ and let $D \in \mathcal{S}(B,I)$. Then, as $I \cap B \subseteq X \subseteq B$,
	\begin{align*}
	    \varphi \circ \phi(X) & = \varphi(X+I) = (X+I) \cap B = (X \cap B) + (I \cap B) = X ,\\
	    \phi \circ \varphi(D) & = \phi(D \cap B)  = (D \cap B) + I = D,
	\end{align*}
 where the last equality follows from the fact that $I$ has codimension 1 in $D$ and $I \nsubseteq B$.
	It follows that there is a 1-1 correspondence between the members of $\mathcal{S}(B,I)$ and $\mathcal{T}(B,I)$.
	Therefore, $|\mathcal{S}(B,I)| =|\mathcal{T}(B,I)| = \qbin{k-t}{1}{q}$, since this counts the number of $(t+1)$-dimensional subspaces of $B$ that contain $I \cap B$. 
    We claim now that the $\mathcal{S}(B,I)$ are disjoint. Let $B_1, B_2 \in \mathcal{B}_I$ and let 
    $D \in \mathcal{S}(B_1,I) \cap \mathcal{S}(B_2,I)$. Then $B_1$ and $B_2$ both each meet $D$ in spaces of dimension $t+1$, say $A_1 = B_1 \cap D$ and $A_2 = B_2 \cap D$.
  Then $A_1 \cap A_2  = B_1 \cap B_2 \cap D $ has dimension $t$, being an intersection of two spaces of codimension 1 in $D$. This contradicts the fact that every $t$-dimensional subspace of $E$ is contained in a unique block.
    Therefore, 
    \[\left| \bigcup_{B \in \mathcal{B}_I} \mathcal{S}(I,B) \right| 
    = \sum_{B \in \mathcal{B}_I} \left|  \mathcal{S}(I,B) \right| 
    = |\mathcal{B}_I| \qbin{k-t}{1}{q} =   \qbin{t+1}{1}{q} \qbin{k-t}{1}{q}.\]
    The number of $(t+2)$-dimensional subspaces $C$ that contain $I$ and do not meet any block $B \in \mathcal{B}_I$ in a space of dimension $t+1$ is thus
    \[ \qbin{n-t-1}{1}{q}  - \qbin{t+1}{1}{q} \qbin{k-t}{1}{q}.\]
    By Theorem \ref{design_independent}, there are exactly $q^{k-t}\qbin{n-k}{1}{q}$ different $(t+1)$-dimensional 
    independent spaces that contain $A$.
	It follows that 
	$$N(A)=q^{k-t}\qbin{n-k}{1}{q}\left(\qbin{n-t-1}{1}{q}- \qbin{t+1}{1}{q} \qbin{k-t}{1}{q}\right),$$
	which is independent of our choice of $A$ of dimension $t$.
	Now for a fixed circuit $C \in {\mathcal{C}_{t+2}}$ containing $A$ there are 
	\[ \qbin{(t+2)-t}{(t+1)-t}{q}=\qbin{2}{1}{q}=\frac{q^2-1}{q-1}=q+1\]    
	independent $(t+1)$-spaces containing $A$ and contained in $C$. So we have
	$$N(A)= (q+1) |\{C \in {\mathcal{C}_{t+2}} : A \subseteq C\}|=(q+1) \lambda(A).$$
	We conclude that
	\[ \lambda_{\mathcal{C}_{t+2}}=\lambda(A)=q^{k-t}\qbin{n-k}{1}{q}\left( \qbin{n-t-1}{1}{q}-\qbin{t+1}{1}{q}\qbin{k-t}{1}{q}\right)\frac{1}{q+1}. \]
\end{proof}

\begin{remark}
{
For the putative $q$-analogue of the Fano plane, the dual of the construction in Theorem \ref{design_circuits-2} was described in \cite[Thm.~4.1]{KP}. This Theorem states that the existence of a $2$-$(7,3,1;q)$ design (i.e., a $S(2,3,7;q)$ Steiner system) implies the existence of a $2$-$(7,3,q^4;q)$ design and is proved by showing that the dual spaces of some subspaces described by the authors as `of type $4_0$' form a $2$-$(7,3,q^4;q)$ design. Spaces of type $4_0$ are $4$-spaces that do not contain any block of the original design. Indeed, taking $k=3$ and $t=2$ in Proposition \ref{prop:circuits} shows that the spaces of type $4_0$ are exactly the $4$-circuits of the $q$-PMD induced by the $S(2,3,7;q)$ Steiner system. They form a $2$-$(7,4,q^6+q^4;q)$ design by Theorem \ref{design_circuits-2} and the corresponding dual design would have parameters $2$-$(7,3,q^4;q)$ by Definition \ref{lem:dualdesign}.}
\end{remark}

The admissibility of design parameters (see Lemma \ref{intersection_numbers}) plays an important role on the question of existence of subspace designs. In the following corollary we give admissibility conditions on the parameters of the design presented in Theorem \ref{design_circuits-2} arising from an $STS(n,q)$. In order to arrive at {\em normalised} parameters in all cases (i.e. those for which $2k\leq n$ and $2\lambda \leq \lambda_{\max}$, where $\lambda_{\max}$ is the maximum possible value corresponding to an admissable parameter set $t$-$(n,k,\lambda_{\max})$), we also calculate the parameters of the supplementary and dual subspace designs.

\begin{corollary}\label{cor:STS}
	If a $q$-Steiner triple system $STS(n;q)$ exists, then there exist $2$-$(n,k,\lambda;q)$ designs with the following parameters.
	\begin{enumerate}
		\item[(1)] $k=4$, $\displaystyle \lambda=q^4\frac{(q^{n-3}-1)(q^{n-6}-1)}{(q^2-1)(q-1)}$,
		\label{cor:STSpart1}
		\item[(2)] $k=4$, $\displaystyle \lambda= \frac{(q^{n-3}-1)(q^4-1)}{(q^2-1)(q-1)}$,\label{cor:STSpart2}
		\item[(3)] $k=n-4$, $\displaystyle\lambda= \qbin{n-3}{3}{q}$,\label{cor:STSpart3}
		\item[(4)] $k=n-4$, $\displaystyle\lambda= q^4\qbin{n-3}{4}{q}$.\label{cor:STSpart4}
	\end{enumerate}
	{In this case, the design with the parameters of (2) is the supplementary design of the design with parameters (1), the design of (3) is the dual design of the design with parameters (2), } and the design of (4) is the dual of the design with parameters (1).
	Moreover, the parameters of the designs listed above 
	are admissible if and only if $n\equiv 0,1,3,4 \mod 6$.
\end{corollary}
\begin{proof}
	As a special case of Theorem \ref{design_circuits-2}, if an STS$(n;q)$ exists then a $2$-$(n,4,\lambda;q)$ design $\mathcal{ D}$ exists for 
	\begin{eqnarray*}
		\lambda & = & q\qbin{n-3}{1}{q}\left( \qbin{n-3}{1}{q}-\qbin{3}{1}{q}\right)\frac{1}{q+1}\\
		& = & q^4\frac{(q^{n-3}-1)(q^{n-6}-1)}{(q^2-1)(q-1)}.
	\end{eqnarray*}	
	It is straightforward to verify that the supplementary design ${\mathcal D}'$ of $\mathcal{D}$ has parameters as given in (2), 
	that the dual design of $\mathcal{D}'$ has parameters as given in (3), and that the dual design of $\mathcal{D}$ has parameters as given in (4).
	Clearly, the parameters of (1)-(3) are admissable or not admissable simultaneously. Therefore we need only check for admissability of one parameter set. To this end we only consider the parameters in (3). Note that these parameters are admissable if and only if
$\lambda_2=\lambda$, $\lambda_1$ and $\lambda_0$ (the number of blocks of the design) are all
non negative integers.
	From Lemma the \ref{intersection_numbers}, the intersection number $\lambda_1$ corresponding to the parameters $2$-$(n,n-4,\lambda=\qbin{n-3}{3}{q};q)$ of (3) is 
	\begin{eqnarray*}
		\lambda_{1} & = & 
		\qbin{n-3}{3}{q}\qbin{n-1}{4}{q}\qbin{n-2}{4}{q}^{-1} =
		\frac{q^{n-1}-1}{q^3-1}\frac{q^{n-3}-1}{q^2-1}\frac{q^{n-4}-1}{q-1}.
	\end{eqnarray*}
	If $n\equiv 0,1,3,4 \mod 6$ such that $n>6$, we see that $\lambda,\lambda_{1}$ are positive integers and hence the parameters of (3)
	are admissable.
	Conversely, assume that the parameters of (3) are admissable. Then in particular $\lambda_1$ is an integer, which holds if and only if either
	  \begin{align*}
	  \begin{cases} 
      \displaystyle\sum_{i=0}^{2r-2}q^i\sum_{j=0}^{2r-4}q^j\sum_{t=0}^{r-3}q^{2t}\equiv 0 \mod q^2+q+1  \text{ and }n=2r \text{ for some } r\in \mathbb{Z}, \\
      \displaystyle \sum_{i=0}^{2r-1}q^i\sum_{j=0}^{r-1}q^{2j}\sum_{t=0}^{2r-4}q^t\equiv 0 \mod q^2+q+1  \text{ and }n=2r+1 \text{ for some } r\in \mathbb{Z}.
      \end{cases}
	  \end{align*}
	     For $n=2r$, if $\displaystyle\sum_{i=0}^{2r-2}q^i$ is divisible by $q^2+q+1$, then $2r-1\equiv 0 \mod 3$, which implies $n\equiv 1 \mod 3$. Then since $n\equiv 1 \mod 3$ and $n\equiv 0 \mod 2$ we obtain $n\equiv 4 \mod 6$. If $\displaystyle\sum_{j=0}^{2r-4}q^j$ is divisible by $q^2+q+1$, then $2r-3\equiv 0 \mod 3$, which implies $n=2r\equiv 0 \mod 3$. Both $n\equiv 0 \mod 3$ and $n\equiv 0 \mod 2$ implies $n\equiv 0 \mod 6$. For $n=2r+1$, if $\displaystyle\sum_{i=0}^{2r-1}q^i$ is divisible by $q^2+q+1$ then $2r-2\equiv 0 \mod 3$, which implies $n=2r+1\equiv 0 \mod 3$. Both $n\equiv 1 \mod 2$ and $n\equiv 0 \mod 3$ implies $n\equiv 3 \mod 6$. And finally, if $\displaystyle\sum_{t=0}^{2r-4}q^t$ is divisible by $q^2+2+1$ then $2r-3\equiv 0 \mod 3$, which implies $n=2r+1\equiv 1 \mod 3$. Both $n\equiv 1 \mod 2$ and $n\equiv 1 \mod 3$ implies $n\equiv 1 \mod 6$. On the other hand, when $n\equiv 2 \mod 6$, that is $n=6m+2$ for some $m\in \mathbb{Z}$, then 
	\begin{align*}
	  \lambda_1 &=  \frac{q^{6m+1}-1}{q^3-1}\frac{q^{6m-1}-1}{q^2-1}\frac{q^{6m-2}-1}{q-1}.
	\end{align*} 
	Hence, $\lambda_1\in \mathbb{Z}$ if and only if $\displaystyle \sum_{i=0}^{6m}q^i\sum_{j=0}^{6m-2}q^j\sum_{t=0}^{3m-2}q^{2t}$ is divisible by $q^2+q+1$. Since $\displaystyle\sum_{i=0}^{6m}q^i\sum_{j=0}^{6m-2}q^j \equiv q+1 \mod q^2+q+1$ and $\gcd(q+1,q^2+q+1)=1$, $\lambda_1\in \mathbb{Z}$ if and only if $\displaystyle \sum_{t=0}^{3m-2}q^{2t}$ is divisible by $q^2+q+1$, which yields a contradiction.  
Similarly,	when $n\equiv 5 \mod 6$, i.e. $n=6m+5$ for some $m\in \mathbb{Z}$ we have
\begin{align*}
 \lambda_{1}&=\frac{q^{6m+4}-1}{q^3-1}\frac{q^{6m+2}-1}{q^2-1}\frac{q^{6m+1}-1}{q-1}.   
\end{align*} Note that $\lambda_1\in \mathbb{Z}$ if and only if $\displaystyle \sum_{t=0}^{3m+1}q^{2t}\sum_{i=0}^{6m+1}q^i\sum_{j=0}^{6m}q^j \equiv 0 \mod q^2+q+1$. Since $\displaystyle \sum_{i=0}^{6m+1}q^i\sum_{j=0}^{6m}q^j\equiv q+1 \mod q^2+q+1$, and $\gcd(q+1, q^2+q+1)=1$, $\lambda_1\in \mathbb{Z}$ if and only if $\displaystyle\sum_{t=0}^{3m+1}q^{2t}$ is divisible by $q^2+q+1$, which yields a contradiction.
\end{proof}

Table~\ref{explicit_param13} shows the parameters that we obtain from the Steiner system $STS(13;q)$ and Corollary \ref{cor:STS}. {In particular, the normalized form of these parameters is $2$-$(13,4,5115;2)$. Of course the other two parameter sets are immediately implied by this one.} Table~\ref{explicit_param} summarizes the parameters of subspace designs whose existence would be implied by the existence of the $q$-Fano plane.

\begin{remark}\label{rem:newdesign}
	In the literature, the only known Steiner triple systems found are those with parameters $STS(13;2)$.
		The existence of such $STS(13;2)$ Steiner triple systems implies, via Corollary \ref{cor:STS}, the existence of new subspace designs with parameters as shown in Table~\ref{explicit_param13}.

		Moreover, for $q=2,3$ and $n=7$, Corollary \ref{cor:STS} shows that the existence of the $q$-Fano plane implies the existence of $2$-$(7,3,15;2)$ and $2$-$(7,3,40;3)$ designs (see also Table~\ref{explicit_param}). Designs with these parameters have actually been found \cite{braun2005some}. However, for $q\geq 4$, there is no information on the existence of designs with parameters $2$-$ (7,3,\frac{q^4-1}{q-1};q)$, which would arise from the $q$-Fano plano over $\mathbb{F}_q$. 
\end{remark} 

\begin{table}[h!]
	\begin{center}
		\caption{Parameters of the new designs in Corollary \ref{cor:STS} from an $STS(13;2)$.}\label{explicit_param13}
		\begin{tabular}{l|p{50mm}}
			\hline
			$q=2$ & $2$-$(13,4,692912;2)$ 
			\newline $2$-$(13,4,5115;2)$ 
			\newline $2$-$(13,9,6347715;2)$ 
			\newline  $2$-$(13,9,859903792;2)$ \\
			\hline
		\end{tabular}
	\end{center}
	
\end{table}

\begin{table}[h!]
	\begin{center}
		\caption{Parameters of the designs of Corollary \ref{cor:STS} from a putative $STS(7;q)$.}\label{explicit_param}
		\begin{tabular}{l|p{55mm}}
			\hline
			$q=2$ & $2$-$(7,4,80;2)$ \newline $2$-$(7,4,75;2)$ \newline $2$-$(7,3,15;2)$ \hfill\cite{braun2005some}\cite[Table 1]{braun2018q} 
			\newline  $2$-$(7,3,16;2)$
			\\
			\hline
			$q=3$ & $2$-$(7,4,810;3)$ 
			\newline $2$-$(7,4,400;3)$ 
			\newline $2$-$(7,3,40;3)$\hfill\cite{braun2005some}\cite[Table 2]{braun2018q} 
			\newline $2$-$(7,3,81;3)$
			\\
			\hline
			$q=4$ & $2$-$(7,4,4352;4)$ \newline $2$-$(7,4,1445;4)$ \newline $2$-$(7,3,85;4)$
			\newline $2$-$(7,3,256;4)$
			\\
			\hline
			$q=5$ & $2$-$(7,4,16250;5)$ \newline $2$-$(7,4,4056;5)$ \newline $2$-$(7,3,156;5)$ 
			\newline  $2$-$(7,3,625;5)$
			\\
			\hline
		\end{tabular}
	\end{center}
\end{table}

\begin{remark}
  Let $n,k,t$ be positive integers satisfying $n \geq k \geq t+1$ and suppose that the parameters $S(t,k,n;q)$ are admissable. One may ask the question as to whether this implies that the parameters 
  \begin{enumerate}
      \item[(1)] $t$-$\left(n,t+1,\qbin{k-t}{1}{q}\right)$ (Theorem \ref{design_circuits}),
      \item[(2)] $t$-$\displaystyle\left(n,t+2,q^{k-t}\qbin{n-k}{1}{q}\left( \qbin{n-t-1}{1}{q}-\qbin{k-t}{1}{q}\qbin{t+1}{1}{q}\right)\frac{1}{q+1} \right)$ (Theorem \ref{design_circuits-2}),
  \end{enumerate}
  are admissable. 
  In the case that $t=2$, from \cite{buratti+} we have that a $\mathcal{S}(2,k,n;q)$ Steiner system exists only if $n\equiv 1,k \mod k(k-1)$. 
  We have found by a computer check that for the case $t=2$, if $q\leq 11$, for $3\leq k \leq 20$, and $n\in \{ 1+ik(k-1),k +ik(k-1):1\leq i \leq 40\}$ then the parameters of (2) are admissable. 
  Similarly, for the cases $t=3,4$, $q \in \{2,3,4,5,7,8,9,11,13\}; 
  k \in \{t+2,...,15\} , n\in \{k+3,...,300\}$, we have found that the parameters of (2) are admissable whenever the $\mathcal{S}(2,k,n;q)$ parameters are admissable, while the converse does not hold.
\end{remark}
While the experimental evidence appears to suggest that the admissability of the parameters of an $\mathcal{S}(2,k,n;q)$ Steiner system implies the admissability of the circuit designs constructed in this paper, calculations to prove this are rather formidable. We give a proof for the case $t=2$ regarding the  $(t+1)$-dimensional circuit designs.
\begin{proposition}
    Let $n,k$ be positive integers satisfying $n \geq k \geq 3$. If the parameters $S(2,k,n;q)$ are admissable then the parameters $2$-$\left(n,3,\qbin{k-2}{1}{q}\right)$ are admissable.
\end{proposition}
\begin{proof}
   Suppose that the parameters $S(t,k,n;q)$ are admissable, so that
   $n\equiv 1,k \mod k(k-1)$.
   The parameters $2$-$\left(n,3,\qbin{k-2}{1}{q}\right)$ are admissable if and only if, for $i =0,1$ we have:
   \begin{equation*}\label{eq:ind}
   B_i:=\qbin{k-2}{1}{q}\qbin{n-i}{3-i}{q}\qbin{n-2}{1}{q}^{-1} 
   = \frac{q^{k-2}-1}{q-1}\prod_{j=i}^1  \frac{q^{n-j}-1}{q^{3-j}-1}\in \mathbb{Z}.
   \end{equation*}
   If $n-1$ is even then $q^2-1$ divides $q^{n-1}-1$. On the other hand, since $k-1$ divides $n-1$, if $n-1$ is odd then $k-2$ is even and so $q^2-1$ divides $q^{k-2}-1$. Clearly, in either case
   $B_1$ is a positive integer. Now consider 
   \[
   B_0=\frac{q^{k-2}-1}{q-1} \frac{q^{n}-1}{q^{3}-1}\frac{q^{n-1}-1}{q^{2}-1} =\frac{q^{n}-1}{q^{3}-1} B_1 .
   \] 
   If $k\equiv 0 \mod 3$ and $k \equiv 1 \mod 2$ then $(q^3-1)(q^2-1) | (q^k-1)(q^{k-1}-1)$, which divides $(q^n-1)(q^{n-1}-1)$ by the admissability of
   $S(t,k,n;q)$. Therefore, $B_0 \in {\mathbb Z}$. 
   If $k\equiv 0 \mod 3$ and $k \equiv 0 \mod 2$ then 
   $(q^3-1)(q^2-1) | (q^n-1)(q^{k-2}-1)$ and so again $B_0$ must be an integer.
   If $k \equiv 2 \mod 3$ then $(q^3-1) | (q^{k-2}-1)$ and clearly 
   $(q-1)(q^2-1)|(q^n-1)(q^{n-1}-1)$, so that $B_0 \in {\mathbb Z}$.
   Finally, suppose now that $k \equiv 1 \mod 3$, so that $n\equiv 1 \mod 3$. If $k\equiv 0 \mod 2$ then $(q^3-1)(q^2-1) | (q^{n-1}-1)(q^{k-2}-1)$ and so $B_0 \in {\mathbb Z}$. If $k \equiv 1 \mod 2$ we can consider the parameters of the supplementary design (i.e. the design of Theorem \ref{design_independent}), which are admissible if and only if
   \[
   C_0=\frac{q^{n-k}-1}{q-1} \frac{q^{n}-1}{q^{3}-1}\frac{q^{n-1}-1}{q^{2}-1} \text{ and }
   C_1=\frac{q^{n-k}-1}{q-1} \frac{q^{n-1}-1}{q^{2}-1} 
   \]
   are both positive integers. We have $n-k \equiv 0 \mod 3$, so that $C_0\in {\mathbb Z}$ and $n-k \equiv 0 \mod 2$, so that $C_1\in {\mathbb Z}$. It follows that the parameters of the supplementary design are admissable in the final case $k \equiv 1 \mod 3$ and  $k\equiv 1 \mod 2$.
\end{proof}

\subsection{The automorphism group of designs from $q$-PMDs}
For the subspace designs constructed in Theorems~\ref{design_independent},~\ref{design_circuits}, and~\ref{design_circuits-2}, we show that their automorphism groups are isomorphic to the automorphism group of the original Steiner system. Since the designs in Theorems \ref{design_independent} and \ref{design_circuits} are supplementary to each other and moreover the constructions of 
the circuits in Theorem \ref{design_circuits-2} are obtained by the independent spaces in Theorem \ref{design_independent}, we only consider the automorphism groups of the designs in Theorem 
\ref{design_independent} and~\ref{design_circuits-2}, respectively.

\begin{theorem}
	Let $\mathcal{S}$ be an $S(t,k,n; q)$ $q$-Steiner system. Then 
	\begin{enumerate} 
		\item The automorphism group of the subspace design obtained in Theorem~\ref{design_independent} from $\mathcal{S}$ is isomorphic to the automorphism group of $\mathcal{S}$.\label{autopart1}
		\item The automorphism group of the subspace design obtained in Theorem~\ref{design_circuits-2} from $\mathcal{S}$ is isomorphic to the automorphism group of $\mathcal{S}$.\label{autopart2}
	\end{enumerate}
\end{theorem}
\begin{proof}
	Let $\mathcal{B}$ denote the blocks of $\mathcal{S}$. 
	Let $\mathcal{I}_{t+1}$ be the set of independent spaces of dimension $t+1$ and let $\mathcal{C}_{t+2}$ be the set of circuits of dimension $t+2$
	of the $q$-PMD arising from $\mathcal{S}$.
	
	\ref{autopart1}.  No member of $\mathcal{I}_{t+1}$ is contained in a block of $\mathcal{B}$. 
	Let $\phi$ be an automorphism of $\mathcal{S}$. Given an independent space $I\in \mathcal{I}_{t+1}$, we claim that the image $\phi(I)$ is also an independent space. 
	Since $\phi \in \mathrm {Aut}(E,{\mathcal {I}}_{t+1})$, then $\dim \phi(I)=t+1$. Moreover, $\phi(I)$ cannot be contained in a block $B$ of $\mathcal{B}$, because otherwise $I \subseteq \phi^{-1}(B)\in \mathcal{B}$, which is a contradiction. Therefore, $\phi(\mathcal{I}_{t+1}) = \mathcal{I}_{t+1}$ and so $\phi \in \rm{Aut(E,\mathcal{I}_{t+1})}$.

	Conversely, let $\phi$ be an automorphism of the subspace design with blocks $\mathcal{I}_{t+1}$. We will show that $\phi(B)\in \mathcal{B}$ for all $B\in \mathcal{B}$.
	Let $B \in \mathcal{B}$ and let $A$ be a $t$-dimensional subspace of $B$ such that $\phi(A)\neq A$. Note that if no such space exists, then $\phi(B)=B$, hence $\phi(B)$ is a block.  
	We will denote by $B_A=B$ the unique block containing $A$ and by $B_{\phi(A)}$ the unique block in $\mathcal{B}$ containing $\phi(A)$. Now assume that $\phi(B_A)$ is not a block, 
	which in particular implies that there exists a one-dimensional subspace $x\subseteq \phi(B_A)$, $x\nsubseteq B_{\phi(A)}$. By considering the independent space construction in Theorem \ref{design_independent}, we claim that the set $I_A=\phi(A)+x$ is independent, since it has dimension $t+1$ and is not contained in a block. Indeed, if it were contained in a block $B'\neq B_{\phi(A)}$, then $\phi(A)$ would be contained in 
	two different blocks and this would contradict the fact that $\mathcal{S}$ is a Steiner system. Finally, note that since $\mathcal{I}_{t+1}$ is the set of blocks of a subspace design, there are $\lambda_I$ independent spaces 
	$I_1, \ldots ,I_{\lambda_I}$ each of which contains $A$. It follows that for each $i$, $\phi(A)\subseteq \phi(I_i)$ and since each $I_i \nsubseteq B_A$, we have $\phi(I_i)\not \subseteq \phi(B_A)$. Since $I_A \subseteq \phi(B_A)$, it follows that $I_A$ is different from each of the 
	subspaces $I_i$. However, in that case $\phi(A)$ is contained in $\lambda_I+1$ independent subspaces in $\mathcal{I}_{t+1}$, yielding a contradiction. It follows that $\phi(B) \in \mathcal {B}$ for each $B \in \mathcal{B}$ and so the result follows.

	\noindent
	\ref{autopart2}. Let $\phi$ be an automorphism of $\mathcal{S}$. Let $C\in \mathcal{C}_{t+2}$. If $\phi(C)$ is not a circuit of dimension $t+2$, there exists a $(t+1)$-dimensional subspace $I'\subseteq \phi(C)$ that is contained in a block $B$ of $\mathcal{B}$. Then $\phi^{-1}(I')$ is a $(t+1)$-subspace of $C$ such that $\phi^{-1}(I')\subseteq \phi^{-1}(B)$. This contradicts the fact that $C$ is a circuit of dimension $t+2$. 
	It follows that $\phi$ is an automorphism of $(E,\mathcal{C}_{t+2})$.
	
	Conversely, let $\phi \in {\rm Aut}(E,\mathcal{C}_{t+2})$. We claim that $\phi$ is also an automorphism of $(E,\mathcal{I}_{t+1})$. Let $I \in \mathcal{I}_{t+1}$ and let $C \in \mathcal{C}_{t+2}$ such that $I \subseteq C$. Then $\phi(C) \in \mathcal{C}_{t+2}$ and so is a circuit of dimension $t+2$ that contains the $(t+1)$-dimensional space $\phi(I)$. It follows that $\phi(I)$ is independent and so $\phi(I) \in \mathcal{I}_{t+1}$. It now follows from \ref{autopart1} that $\phi $ is an automorphism of $\mathcal{S}$.
\end{proof}

\begin{remark}\label{rem:autgroup}
	Subspace designs with parameters $2$-$(7,3,15;2)$ and $2$-$(7,3,40;3)$ were found by computer search in \cite{braun2005systematic}, applying the Kramer-Mesner method and under the assumption that their automorphism groups contain a Singer cycle. In Table $1$, we see that subspace designs with the same parameters appear, with such designs arising from an $STS(7;2)$ $q$-Steiner triple system. However, the designs of \cite{braun2005systematic}
	could not be constructed by the methods of this paper, as then
	their automorphism groups would be isomorphic to that of the $q$-Fano plane, which has automorphism group of order at most $2$~\cite{BraunKier,KierKurz}. 
\end{remark}

\section{Acknowledgements}
{The authors are indebted to Ferdinand Ihringer for pointing out a counting error in Theorem \ref{design_circuits-2} in an earlier version of this paper, as well as to Ragnar Freij-Hollanti for his remarks on Proposition \ref{prop:classifysteinerflats} and Theorem \ref{thm:steinertoflats}. The authors wish to thank the anonymous referees for their helpful comments.}
This paper is the product of a collaboration that was initiated at the Women in Numbers Europe (WIN-E3) conference,
held in Rennes, August 26-30, 2019. The authors are very grateful to the organisers: Sorina Ionica, Holly Krieger, and Elisa Lorenzo Garc\'ia, for facilitating their participation at this workshop, which was supported by the Henri Lebesgue Center, the Association for Women in Mathematics (AWM) and the Clay Mathematics Institute (CMI). The fifth author is supported by the Swiss National Science Foundation grant n. 188430\footnote{E. Sa\c{c}\i kara was supported by the Swiss Confederation through the Swiss Government Excellence Scholarship no: 2019.0413 between September 2019 and August 2020.}. \\

\bibliographystyle{spmpsci}
\bibliography{References}

\end{document}